\documentclass[12pt,leqno]{amsart}  
\usepackage{amsmath,amsthm,amssymb,amsxtra,amstext}

\usepackage[colorlinks,citecolor=red,pagebackref,hypertexnames=false]{hyperref}

\newtheorem{proposition}{Proposition}[section]
\newtheorem{theorem}[proposition]{Theorem}
\newtheorem{lemma}[proposition]{Lemma}

\newtheorem{conjecture}[proposition]{Conjecture}

\newtheorem{definition}{Definition}[section]

\usepackage{mathtools}
\mathtoolsset{showonlyrefs,showmanualtags}

\usepackage{hyperref} %,hypertexnames=false,colorlinks,[pagebackref]
\hypersetup{
    colorlinks=true,       % false: boxed links; true: colored links
    linkcolor=blue,          % color of internal links
    citecolor=magenta,        % color of links to bibliography
    filecolor=magenta,      % color of file links
    urlcolor=cyan           % color of external links
}

\usepackage[msc-links]{amsrefs}  

\allowdisplaybreaks
\swapnumbers

\numberwithin{equation}{section}

\title[$ L ^{1}$ Dichotomy for the Discrepancy Function]{Dichotomy Results for the $ L ^{1}$ Norm of the Discrepancy Function} 
\author{Gagik Amirkhanyan} 
\address{School of Mathematics, Georgia Institute of Technology, Atlanta GA 30332, USA}
\email{gagik@math.gatech.edu}

\author{Dmitriy Bilyk}
\address{School of Mathematics, University of Minnesota, Minneapolis MN 55455, USA} \email{dbilyk@math.umn.edu}
\thanks{} 

\author{Michael T  Lacey}   %  can use \and  

\address{School of Mathematics, Georgia Institute of Technology, Atlanta GA 30332, USA}
\email {lacey@math.gatech.edu}
\thanks{}


\begin{document}

\maketitle

\begin{abstract} It is a well-known conjecture in the theory of irregularities of distribution that the $L^1$ norm of the discrepancy function of an $N$-point set  satisfies the same asymptotic  lower bounds as its $L^2$ norm. In dimension $d=2$ this fact has been established by Hal\'{a}sz, while in higher dimensions the problem is wide open. In this note, we establish a series of dichotomy-type results which state that if the $L^1$ norm of the discrepancy function is too small (smaller than the conjectural bound), then the discrepancy function has to be large in some other function space.

\end{abstract}

%%%%%%%%%%%%%%%%%%%%%%%%%%%%%% SECTION  SECTION SECTION
%%%%%%%%%%%%%%%%%%%%%%%%%%%%%% SECTION  SECTION SECTION 
\section{Introduction} %\label{s:}

\subsection{Preliminaries}

For integers $ d\ge 2$, and $ N\ge 1$, let $ \mathcal P_N  \subset  [0,1] ^{d} $ be a finite point set  with cardinality $  \sharp  \mathcal P_N=N $. Define the associated discrepancy function by 
\begin{equation}  \label{e:discrep}
D_N(x)=\sharp( \mathcal P_N \cap [ 0, x))-N \lvert  [ 0, x) \rvert,
\end{equation}
where  $ x= (x_1 ,\dotsc, x_d)$ and $[ 0, x)=\prod _{j=1}^d [0,x_j)$ is a rectangle with antipodal corners at  $ 0$ and $ x$, and $|\cdot |$ stands for the $d$-dimensional Lebesgue measure.
The dependence upon the selection of points $\mathcal P_N$ will be suppressed, as we are interested  in bounds 
that are  only a function of $N= \sharp \mathcal P_N$.  The discrepancy function $D_N$ measures equidistribution of $\mathcal P_N$: a  set of points is \emph{well-distributed} if  $ D_N$ is small in some appropriate function space.

It is a basic fact of the theory of irregularities of distribution that relevant norms of this function  in dimensions $2$ and higher must tend to infinity as $N$ grows. 
The classic results are due to   Roth \cite{MR0066435} in the case of the $ L ^{2}$ norm
and  Schmidt \cite{MR0491574} for   $L ^{p} $, $ 1< p < 2$:
%%%%%%%%%%%%%%%%%%%%%%%%%%%%%% THEOREM THEOREM THEOREM
\begin{theorem}\label{t.DP} 
For $1<p<\infty$ and any collection of points $ \mathcal P_N\subset [0,1] ^{d}$, we have% the 
%estimates 
\begin{equation}\label{e:DP}
\lVert  D_N \rVert _{ p} \gtrsim (\log N) ^{(d-1)/2}\,. 
\end{equation}
Moreover, we have the endpoint estimate
\begin{equation}\label{e:DPend}
\lVert  D_N \rVert _{ L (\log L )  ^{(d-2)/2}} 
\gtrsim (\log N) ^{(d-1)/2}.
\end{equation}
\end{theorem}
%%%%%%%%%%%%%%%%%%%%%%%%%%%%%% THEOREM THEOREM THEOREM
In dimension $ d=2$ the  $L^1$ endpoint estimate  above  was established by  Hal{\'a}sz \cite{MR637361}, while its  Orlicz space generalization  for dimensions $ d\ge 3$ is due to the last author   \cite{MR2419612} (notice that, when $d=2$, we have $L (\log L )  ^{(d-2)/2} = L^1$).

The symbol ``$\gtrsim$" in this paper stands for ``greater than a constant multiple of", and the implied constant may depend on the dimension, the function space, but {\it{not}} on the configuration $\mathcal P_N$ or the number of points $N$. $A\simeq B$ means $A\lesssim B \lesssim A$.

Estimate \eqref{e:DP} is sharp, i.e. there exist sets $ \mathcal P_N$ that meet the $ L ^{p}$  bounds \eqref{e:DP} in all dimensions.  This remarkable fact is established by beautiful and quite non-trivial constructions of point distributions $ \mathcal P_N$. We refer the reader to one of the very good references \cite{MR903025,MR2683394,MR1470456}  for 
more information about low-discrepancy sets, which is an important complement to the theme of this note.  

The subject of our paper is the $ L ^{1}$ endpoint. Hal\'{a}sz's original argument yields the following very weak extension to higher dimensions.
\begin{theorem}\label{t:DL1}
In all dimensions $d\ge 3$, we have \begin{equation}\label{e:DL1}
\| D_N \|_1 \gtrsim \log^{1/2} N. \end{equation}
\end{theorem}
%that is , and the following conjectural improvement of \eqref{e:DPend}. 
No improvements of \eqref{e:DL1} have been obtained thus far -- embarrassingly, it is not even known whether the $L^1$ norm of $D_N$ should grow as the dimension increases. It is widely believed that the correct bound for the $L^1$ norm matches Roth's $L^2$ estimates \eqref{e:DP}.

%%%%%%%%%%%%%%%%%%%%%%%%%%%%%% CONJECTURE CONJECTURE CONJECTURE
\begin{conjecture}\label{j:halasz} In all  dimensions $ d\ge 3$, the following estimate holds \begin{equation}\label{e:Conj}  \lVert D_N\rVert_{L ^{1} ([0,1]^{d}) } \gtrsim 
(\log N) ^{ (d-1)/2}.\end{equation}  
\end{conjecture}
%%%%%%%%%%%%%%%%%%%%%%%%%%%%%% CONJECTURE CONJECTURE CONJECTURE
\noindent Observe that \eqref{e:DPend} supports this conjecture.\\

\subsection{Main results}

While the conjectural bound \eqref{e:Conj} does not seem accessible at this point, we shall  prove several dichotomy-type results for the $ L  ^{1}$ norm, which essentially say  that either the $ L ^{1}$ norm is large, or some 
larger norm has to be very large.  

We start with a very simple result, valid in all dimensions, which states that if a point distribution has optimally small (according to \eqref{e:DP}) $L^p$ norm of the discrepancy, then it has to satisfy the conjectured $L^1$ estimate \eqref{e:Conj}. In other words, if there exist  sets with $L^1$-discrepancy so small  as to violate Conjecture \ref{j:halasz}, they cannot simultaneously have low $L^p$-discrepancy.

\begin{theorem}\label{t:simple}
Let $p\in(1,\infty)$. For every constant $C_1>0$, there exists $C_2>0$ such that whenever $\mathcal P_N \subset [0,1]^d$ satisfies %\begin{equation}\label{e:simple1} 
$\|D_N \|_p  \le C_1 (\log N) ^{ (d-1)/2}$, %\end{equation}
it implies that 
\begin{equation}\label{e:simple2} \|D_N \|_1  \ge C_2 (\log N) ^{ (d-1)/2}. \end{equation}
\end{theorem}

The next theorem, also true for general dimensions, amplifies this effect. It states that  if the $L^1$-discrepancy fails Conjecture \ref{j:halasz} by a small exponent, then the $L^2$-discrepancy is not just suboptimal, but huge.

%%%%%%%%%%%%%%%%%%%%%%%%%%%%%% THEOREM THEOREM THEOREM
\begin{theorem}\label{t:dichot} 
For all dimensions $ d\ge 3$, there is an $ \epsilon = \epsilon (d)>0$  and $ c = c (d)>0$ such that for all integers $ N\ge 1$,   every $\mathcal P_N \subset [0,1]^d$ satisfies either
\begin{equation*}
\lVert D_N\rVert_{1} \ge (\log N ) ^{(d-1)/2- \epsilon } 
\quad \textup{or} \quad 
\lVert D_N\rVert_{2} \ge \operatorname {exp}(c (\log N) ^{\epsilon })\,. 
\end{equation*}
\end{theorem}
%%%%%%%%%%%%%%%%%%%%%%%%%%%%%% THEOREM THEOREM THEOREM

Thus a putative example of a distribution $ \mathcal P_N$ with $ D_N$ very small in the $ L ^{1}$ norm  
must be very far from extremal in the $ L ^2 $-norm.  
The proof will show that one can take $ \epsilon (d)$ as large as a fixed multiple of $ 1/d$.  
Specializing to the case of dimension $ d=3$, we can replace the $ L ^{2}$ norm above by a much smaller norm.  
  
%%%%%%%%%%%%%%%%%%%%%%%%%%%%%% THEOREM THEOREM THEOREM
\begin{theorem}\label{t:product} In dimension $ d=3$, there holds  
\begin{equation*}
\lVert D_N\rVert_{1}  \cdot  \lVert D_N\rVert_{L (\log L)} \gtrsim (\log N) ^2  \,. 
\end{equation*}
\end{theorem}
%%%%%%%%%%%%%%%%%%%%%%%%%%%%%% THEOREM THEOREM THEOREM 

Unfortunately, this estimate is consistent with a putative distribution $\mathcal P_N $, for which  $ \lVert D_N\rVert_{1} \lesssim  (\log N ) ^{1/2} $.  
The  last theorem of this series addresses possible examples, where $ D_N$ is less that $ (\log N )^{1/2} $ in the  $ L ^{1} $ norm.  

%%%%%%%%%%%%%%%%%%%%%%%%%%%%%% COROLLARY COROLLARY COROLLARY
\begin{theorem}\label{c:} For all dimensions $ d\ge 3$ and all  $ C_1> 0$, there is a $ C_2>0$ so that 
if $ \lVert D_N\rVert_{1} \le C_1 \sqrt {\log N}$, then $ \lVert D_N\rVert_{2} \gtrsim N ^{C_2}$.   
\end{theorem}
%%%%%%%%%%%%%%%%%%%%%%%%%%%%%%  COROLLARY COROLLARY COROLLARY

Finally,   the dichotomies above are of an essentially optimal nature in light of the examples in this 
next result. 

%%%%%%%%%%%%%%%%%%%%%%%%%%%%%% THEOREM THEOREM THEOREM
\begin{theorem}\label{t.example} 
For all dimensions $ d\ge 2$, there is a distribution such that 
\begin{equation*}
\lVert D_N\rVert_{1} \lesssim (\log N ) ^{(d-1)/2} 
\quad \textup{and} \quad 
\lVert D_N\rVert_{2} \gtrsim N^{1/4}.
\end{equation*}
\end{theorem}
%%%%%%%%%%%%%%%%%%%%%%%%%%%%%% THEOREM THEOREM THEOREM

The proofs are based upon the detailed information used to obtain non-trivial improvement in the $ L ^{\infty }$ endpoint estimates 
in \cites{MR2414745,MR2409170}.  We recall the required estimates in the next section and then turn to the proofs of  Theorems \ref{t:simple}--\ref{t.example} in \S\ref{s.proof}.

\section{The Orthogonal Function Method} %\label{s.}

All progress on these universal lower bounds has been based upon  the orthogonal function method, initiated by Roth \cite{MR0066435}, with the modifications of Schmidt \cite{MR0491574}, as 
presented here.  
Denote the family of all dyadic intervals $ I\subset [0,1]$ by $ \mathcal D$.  Each dyadic interval 
$ I$ is the union of two dyadic intervals $ I _{-}$ and $I_+$, each  of exactly half the length of $ I$, representing
the left and right halves of $ I$ respectively.  Define the Haar function associated to $ I$ 
by $ h_I = - \chi _{I _{-}}+ \chi _{I _{+}}$.  Here  and throughout we will use the $ L ^{\infty }$ (rather than $L^2$)
normalization of the  Haar functions.  

In dimension $ d$, the $ d$-fold product $ \mathcal D ^{d}$ is the collection of dyadic intervals in $ [0,1] ^{d}$. 
Given $ R= R_1 \times \cdots \times R_d \in \mathcal D ^{d}$, the Haar function associated with $ R$ is the 
tensor product 
\begin{equation*}
h _{R} (x_1 ,\dotsc, x_d) = \prod _{j=1} ^{d} h _{R_j} (x_j) \,. 
\end{equation*}
These functions are pairwise orthogonal as $ R\in \mathcal D ^{d}$ varies.  

For a $d$-dimensional vector  $ r= (r_1 ,\dotsc, r_d)$ with non-negative integer coordinates %$r_j \in \mathbb N\cup \{0\}$, $j=1,...,d$, 
let $ \mathcal D _{r}$ be the set of those $ R\in \mathcal D ^{d}$ that 
for each coordinate $ 1\le j \le d$, we have $ \lvert  R _{j}\rvert = 2 ^{- r_j} $. These rectangles partition $ [0,1] ^{d}$.  
We call $ f _{r}$ an $ r$-function (a generalized Rademacher function) if for some choice of signs $ \{\varepsilon _{R} \;:\; R\in \mathcal D _{r}\}$, we have 
\begin{equation*}
f _{r} (x) = 
\sum_{R \in \mathcal D _{r}} \varepsilon _R h _{R} (x)\,. 
\end{equation*}

The following is the crucial lemma of the method, see  \cite{MR0066435,MR0491574,MR2817765}.   
Given an  integer $ N$, we set $ n = \lceil 1 + \log_2 N \rceil$, where $\lceil x\rceil$ denotes the smallest integer greater than or equal to $x$. %, and consider a discrepancy function $ D_N$. 

%%%%%%%%%%%%%%%%%%%%%%%%%%%%%% LEMMA LEMMA LEMMA
\begin{lemma}\label{l.r}  
In all dimensions $ d\ge 2$ there is a constant $ c_d > 0 $ such that 
for each $r$ with  $\lvert  r\rvert := \sum_{j=1} ^{d} r_j=n$, there is an $ r$-function $ f _{r}$ with 
$ \langle D_N, f _{r} \rangle \ge c_d$.   
Moreover, for all $ r$-functions there holds  $ \lvert  \langle D_N , f _{r} \rangle\rvert \lesssim N 2 ^{- \lvert  r\rvert }$.  
\end{lemma}
%%%%%%%%%%%%%%%%%%%%%%%%%%%%%% LEMMA LEMMA LEMMA

Heuristically, this lemma quantifies the fact that most of the information about  the discrepancy function is encoded by the Haar coefficients corresponding to boxes $R\in \mathcal D^d$ with volume $|R|\approx 1/N$. The proofs of most known lower bounds for the discrepancy function have been guided by this idea. We briefly outline the argument leading to  \eqref{e:DP}.

For integer vectors $ \vec r \in \mathbb N ^{d}$, let $ f _{\vec r}$ be an $\vec r$-function as in the previous lemma. 
Set 
\begin{equation*}
Z:=  \frac 1 {n ^{(d-1)/2}}\sum_{\vec r \;:\; \lvert  \vec r\rvert=n } f _{\vec r} \,. 
\end{equation*}
It is easy to see  that, due to orthogonality and the fact that  the number of vectors  $\vec r  \in \mathbb N ^{d}$ with $\lvert  \vec r\rvert=n$ is of the order $n^{d-1}$, we have $ \lVert Z\rVert_{2} \simeq 1$. Moreover, it also satisfies $ \lVert Z\rVert_{p} \lesssim 1$ for all $ 1<p < \infty $. This extension can be derived using Littlewood--Paley theory or, as originally done in \cite{MR0491574}, using combinatorial arguments if $p$ is an even integer.
This is enough to establish \eqref{e:DP}:  H\"{o}lder inequality and  Lemma \ref{l.r} yield
\begin{equation}
n^{\frac{d-1}2} \lesssim  \langle D_N, Z\rangle \lesssim  \| D_N \|_p \cdot \| Z\|_{p'} \lesssim \| D_N\|_p.
\end{equation}

The following is a deep exponential-squared  distributional estimate for $ Z$ --  indeed, it is  a key estimate behind the main theorems of \cite{MR2409170} on the $ L ^{\infty }$ norm of the discrepancy function.

%%%%%%%%%%%%%%%%%%%%%%%%%%%%%% THEOREM THEOREM THEOREM
\begin{theorem}\label{t:exp}\cite{MR2409170}*{Theorem 6.1}  There is an absolute constant $0<c<1$, such that in all dimensions $ d\ge 3$, for  $ \epsilon =c/d$ we have   
\begin{equation*}
\lvert  \{  x \;:\; \lvert  Z (x)\rvert > t \}\rvert \lesssim \operatorname {exp}(-c t ^2  ) \,, \qquad  0 < t < c n ^{\frac {1- 2 \epsilon } {4d-2}} 
\end{equation*}
\end{theorem}
%%%%%%%%%%%%%%%%%%%%%%%%%%%%%% THEOREM THEOREM THEOREM

%%%%%%%%%%%%%%%%%%%%%%%%%%%%%% SECTION  SECTION SECTION
%%%%%%%%%%%%%%%%%%%%%%%%%%%%%% SECTION  SECTION SECTION 
\section{Proofs} \label{s.proof}

We now proceed to the proofs of the main theorems. %We start with the easiest form of the dichotomy valid in all dimensions

\begin{proof}[Proof of Theorem~\ref{t:simple}]
Assume that for a given $1<p<\infty$ we have $\| D_N \|_p \le C_1 \big( \log N\big)^{\frac{d-1}{2}}$. The Roth--Schmidt bound \eqref{e:DP}  states that $\| D_N \|_{2p/(p+1)} \ge c_{2p/(p+1)}  \big( \log N\big)^{\frac{d-1}{2}}$. Interpolating between $1$ and $p$ using H\"older's inequality we find that  $\| D_N \|_{2p/(p+1)} \le \| D_N \|_1^{1/2} \, \| D_N \|_p^{1/2}$. Therefore
\begin{equation}
\| D_N \|_1 \ge  \frac{\|D_N\|^2_{2p/(p+1)}}{\|D_N\|_p} \ge \frac{ c^2_{2p/(p+1)}  \big( \log N\big)^{{d-1}} }{ C_1 \big( \log N\big)^{\frac{d-1}{2}} } = C_2 \big( \log N\big)^{\frac{d-1}{2}},
\end{equation}
which proves \eqref{e:simple2} with $C_2 = \frac{ c^2_{2p/(p+1)}}{C_1}$.
\end{proof}

%%%%%%%%%%%%%%%%%%%%%%%%%%%%%% PROOF PROOF PROOF
\begin{proof}[Proof of Theorem~\ref{t:dichot}]  
Set  $ q = n ^{\varepsilon }$, where $ \varepsilon \simeq 1/d$, and define 
\begin{equation*}
Y :=  \frac 1 {n ^{(d-1)/2} q}\sum_{\vec r \;:\; \lvert  \vec r\rvert=n } f _{\vec r} \,. 
\end{equation*}
Then $\| Y\|_p \lesssim q^{-1}$ for $1<p<\infty$. Besides, one has 
$
\langle D_N, Y \rangle \ge c \frac {n ^{(d-1)/2}} {q} 
$. 
But unfortunately $ Y$ is not bounded, preventing an immediate conclusion about the $ L ^{1}$ norm of $ D_N$.  
 
On the other hand, from Theorem~\ref{t:exp} %, for an appropriate $ \varepsilon $, 
 we get
\begin{equation*}
\lvert \{  \lvert  Y\rvert > 1\} \rvert \lesssim \operatorname {exp}(-c q ^2 ) \,. 
\end{equation*}
Using a trilinear  H\"older's inequality, we obtain
\begin{align*}
\int _{\{  \lvert  Y\rvert > 1\}} \lvert  D_N \cdot Y\rvert \; dx &\le \lvert  \{  \lvert  Y\rvert > 1\}\rvert ^{1/4}\lVert Y\rVert_{4}  \lVert D_N\rVert_{2}   
\\
& \lesssim  \operatorname {exp}(-c' q ^2 ) \cdot  q ^{-1}   \lVert D_N\rVert_{2} \,. 
\end{align*}
This last quantity will be at most $ \tfrac 12 \langle D_N, Y \rangle $, if $ \lVert D_N\rVert_{2} \lesssim \operatorname {exp}(c'' q ^2 ) $. 
Then
\begin{align*}
\| D_N \|_1&  \ge \big| \langle D_N, Y \cdot {\bf{1}}_{\{ |Y|\le1\} } \rangle \big| \\
& \ge \langle D_N, Y \rangle - \int _{\{  \lvert  Y\rvert > 1\}} \lvert  D_N \cdot Y\rvert \; dx \ge\frac12  \langle D_N, Y \rangle  \gtrsim n^{\frac{d-1}{2} - \varepsilon}
\end{align*}
and this proves  Theorem~\ref{t:dichot}
\end{proof}
%%%%%%%%%%%%%%%%%%%%%%%%%%%%%% PROOF PROOF PROOF

%%%%%%%%%%%%%%%%%%%%%%%%%%%%%% PROOF PROOF PROOF
\begin{proof}[Proof of Theorem~\ref{t:product}]

Define  
\begin{equation*}
 Y = \frac 1 {\sqrt n} \sum_{j=1} ^{n/2} \sin \Bigl(cn ^{-1/2}  \sum_{\vec r \;:\; r_1 =j}  f_ {\vec r}\Bigr) 
\end{equation*}
where $ 0< c < 1$ is a sufficiently small constant.  

%%%%%%%%%%%%%%%%%%%%%%%%%%%%%% LEMMA LEMMA LEMMA
\begin{lemma}\label{l.Y} 
The following two estimates hold.  First, $ \langle D_N,Y \rangle \gtrsim n  $, and second, 
\begin{gather*}
\mathbb P ( \lvert  Y\rvert > \alpha    ) \lesssim \operatorname {exp}(-c \alpha  ^2 )\, \qquad \alpha >1 \,. 
\end{gather*}
\end{lemma}
%%%%%%%%%%%%%%%%%%%%%%%%%%%%%% LEMMA LEMMA LEMMA

%%%%%%%%%%%%%%%%%%%%%%%%%%%%%% PROOF PROOF PROOF
\begin{proof}
Modify, in a straight forward way,   \cite{MR2419612}*{\S3} to see that for $ c $ sufficiently small, 
\begin{equation*}
\Bigl\langle D_N , \sin \Bigl(cn ^{-1/2}  \sum_{\vec r \;:\; r_1 =j}  f_ {\vec r}\Bigr)  \Bigr\rangle_{} \gtrsim \sqrt n \,, \qquad 1\le j \le n/2 \,. 
\end{equation*}
Sum this over $ j$ to prove the first claim of the Lemma.  

The second claim, the distributional estimate,   is equivalent to the bound $ \lVert Y\rVert_{p} \lesssim C \sqrt p$ for $ 2\le p < \infty $. 
This is  estimate (4.1) in \cite{MR2419612}.
\end{proof}
%%%%%%%%%%%%%%%%%%%%%%%%%%%%%% PROOF PROOF PROOF

  Set 
$ E= \{ \lvert    Y\rvert >  \alpha    \}$, where $ \alpha > 1$ is to be chosen.  
We consider the inner product 
\begin{align*}
c n  \le \langle D_N,  Y\rangle &\le  
\langle D_N, Y \mathbf 1_{E ^{c}} \rangle + \langle D_N, Y \mathbf 1_{E} \rangle 
\\
& \le \alpha \lVert D_N\rVert_{1} + \lVert D_N\rVert_{L (\log L) }  \lVert Y \mathbf 1_{E} \rVert_{\operatorname {exp}(L)}  
\\
& \le \alpha \lVert D_N\rVert_{1} + \alpha ^{-1}  \lVert D_N\rVert_{L (\log L) } 
\,,
\end{align*}
where we have used the duality of the spaces $L (\log L)$ and $\operatorname{exp}(L)$. The last estimate depends upon the calculation 
\begin{align*}
\lVert Y \mathbf 1_{E} \rVert_{ \operatorname {exp}(L )} \simeq 
\sup _{t\ge1}  t \cdot \big| \log \lvert   \{ \lvert  Y\rvert > \max \{t,\alpha\}  \}\rvert \big| ^{-1} \lesssim \sup _{t\ge1} \min \bigg\{ \frac1{t}, \frac{t}{\alpha^2} \bigg\}
\simeq  \alpha ^{-1} \,.  
\end{align*}
Choose $ \alpha ^2  \simeq \lVert D_N\rVert_{L \log L }/ \lVert D_N\rVert_{1} \ge 1$.  We then have 
\begin{equation*}
 n  \lesssim  \lVert D_N\rVert_{L \log L } ^{1/2}  \lVert D_N\rVert_{1} ^{1/2} \,, 
\end{equation*}
and this proves Theorem \ref{t:product}. 
\end{proof}
%%%%%%%%%%%%%%%%%%%%%%%%%%%%%% PROOF PROOF PROOF

%%%%%%%%%%%%%%%%%%%%%%%%%%%%%% PROOF PROOF PROOF
\begin{proof}[Proof of Theorem~\ref{c:}]
Assume that $\| D_N\|_1\le  C_1 \log N$.  We shall utilize the main result of \cite{MR2419612}, namely \eqref{e:DPend}.  
Consider  the probability measure  $ \mathbb P _N$ which is the normalized $ \lvert  D_N\rvert \; dx  $, i.e. $d \mathbb P _N (x) = \frac{ \lvert  D_N (x) \rvert }{ \lVert D_N\rVert_{1}} dx$.   
We see that 
\begin{equation*}
 \int (\log_+ |D_N|)^{\frac {d-2}2} \; d \mathbb P _N (x) \ge \frac {\lVert D_N\rVert_{L (\log L) ^{\frac {d-2}2}} } {  \lVert D_N\rVert_{1} }
 \ge C n ^{\frac {d-2} 2} \,.  
\end{equation*}
It is obvious that $|D_N(x) |\le N$, therefore $ \log |D_N| \le n$. It follows from a Paley--Zygmund-type inequality that for some $c>0$
\begin{equation}\label{e:PZ}
%\int _{ \{\log D_N > C_4n  \}} d \mathbb P _N (x) \gtrsim 1  \,. 
\mathbb P _N \{\log_+ |D_N| > cn  \} \gtrsim 1.
\end{equation}
Indeed, denoting $f=\log_+ |D_N|$ and $\alpha = (d-2)/2$, using Cauchy--Schwarz inequality we get
\begin{align*}
C n^\alpha \le \mathbb E |f|^\alpha & \le \mathbb E |f|^\alpha \mathbf{1}_{\{|f| > cn \}} + \mathbb E |f|^\alpha \mathbf{1}_{\{|f| \le cn \}}\\
& \le \big( \mathbb E |f|^{2\alpha} \big)^{1/2} \cdot \mathbb P_N^{1/2} \{|f| > cn \} + c^\alpha n^\alpha \\ &\le n^\alpha \cdot \big( \mathbb P_N^{1/2} \{|f| > cn \} + c^\alpha),
\end{align*} which yields \eqref{e:PZ} if $c$ is small enough.
From this, using the fact that $\|D_N\|_1 \gtrsim \sqrt{n}$ (Theorem \ref{t:DL1}), we deduce that 
\begin{align*}  
\lVert D_N\rVert_{2} ^2 \gtrsim \int _{ \{\log D_N > cn \}} D_N^2 (x) \; d x \gtrsim  \sqrt n \cdot 
\int _{ \{\log D_N > cn \}} |D_N (x)| \; d \mathbb P _N (x)  \gtrsim N ^{C'},
\end{align*}
which  is the conclusion of Theorem \ref{c:}. 
\end{proof}
%%%%%%%%%%%%%%%%%%%%%%%%%%%%%% PROOF PROOF PROOF

For the last proof we need an additional definition.  

%%%%%%%%%%%%%%%%%%%%%%%%%%%%%%  DEFINITION DEFINITION DEFINITION
\begin{definition}\label{d.net} 
A distribution  $\mathcal P_N$  of $ N = p ^{s}$ points is called a $p$-adic net, if any $p$-adic rectangle
\begin{equation*}
\Delta = \prod _{j=1} ^{d} [m_j p ^{-a_j}, (m_j+1) p ^{-a_j}), \quad  0\le m_j < a_j 
\end{equation*}
of volume $\frac{1}{N}$ contains exactly one point of  $\mathcal P_N$.
\end{definition}
%%%%%%%%%%%%%%%%%%%%%%%%%%%%%%  DEFINITION DEFINITION DEFINITION

For any dimension $ d\ge 2$ and  a prime $p\ge d -1$, there exist nets  with $p^s$ points for all values of $ s\ge 2$.  
One can show that if  $\mathcal P_N$  is a $p$-adic net of $N = p^s$ points, then for any rectangle $R \subset [0, 1]^d$
\begin{equation*}
\big\lvert \sharp( \mathcal P_N  \cap R )   -  |R| N  \big\rvert \leq s^{d-1}.
\end{equation*}
A similar inequality can be obtained for arbitrary $N$.

%%%%%%%%%%%%%%%%%%%%%%%%%%%%%% PROOF PROOF PROOF
\begin{proof}[Proof of Theorem~\ref{t.example}]

Let us take a net  $\mathcal P_N$  with small $L_2$ discrepancy, i.e. 
\begin{equation*}
\lVert D_N\rVert_{2} \lesssim (\log N ) ^{(d-1)/2}.
\end{equation*}
The existence of such nets is well-known \cite{MR1896098,MR2683394}.
Then clearly we also have
$
\lVert D_N\rVert_{1} \lesssim (\log N ) ^{(d-1)/2} 
$.
For $\delta > 0$ we define the cube $Q = [1 - N^{-\delta}, 1]^d$, which lies at the top right corner of $[0, 1]^d$.
As $| Q | = N ^{-\delta d}$ and the distribution $ \mathcal P_N $ is a net, it follows that  $Q$ contains about $N ^{1 -\delta d}$ points of $ \mathcal P_N $.

Let $ \mathcal P_N '$ be a new distribution obtained from  $\mathcal P_N$ by replacing the points inside $Q$ with $(1, 1 ,\dotsc, 1)$ and keeping the points outside $Q$ unchanged.  Let $ D'_N$ be the associated discrepancy function. 
Then $D_N(x) = D'_N(x)$ for $x \not\in Q$, and $D'_N$ has no contribution from the distribution of points inside $Q$. Hence for $x \in Q$
\begin{equation*}
|D_N(x) - D'_N(x) | \lesssim N ^{1 -\delta d}.
\end{equation*}
Because $\mathcal P_N $ is a net,   in a positive proportion of $Q$ we will also have 
\begin{equation*}
|D_N(x) - D'_N(x) | \gtrsim N ^{1 -\delta d}.
\end{equation*}
Therefore we have
\begin{equation*}
\lVert D_N - D'_N\rVert_{1} \simeq N ^{1 - 2 \delta d}
\quad \textup{and} \quad 
\lVert D_N - D'_N\rVert_{2}^{2} \simeq N ^{2 - 3 \delta d}.
\end{equation*}
If we take $\delta = \frac{1}{2d}$, we obtain
\begin{equation*}
\lVert D_N - D'_N\rVert_{1} \simeq 1
\quad \textup{and} \quad 
\lVert D_N - D'_N\rVert_{2} \simeq N^{1/4},
\end{equation*}
which implies that
\begin{equation*}
\lVert D_N'\rVert_{1} \lesssim (\log N ) ^{(d-1)/2} 
\quad \textup{and} \quad 
\lVert D_N'\rVert_{2} \gtrsim N^{1/4}.
\end{equation*}

\end{proof}
%%%%%%%%%%%%%%%%%%%%%%%%%%%%%% PROOF PROOF PROOF

\section{Acknowledgements}

The authors gratefully acknowledge that this research was supported in part by NSF grants DMS-1101519 (D. Bilyk), DMS-0968499 (M. Lacey), a grant from the Simons Foundation \#229596  (M.  Lacey),  and the Australian Research Council through grant ARC-DP120100399 (M. Lacey).

\end{document}